\documentclass[a4paper,oneside,11pt]{article}

\usepackage[utf8]{inputenc}
\usepackage[T1]{fontenc}
\usepackage{textcomp}
\usepackage[english]{babel}
\usepackage[autolanguage]{numprint}
\usepackage{hyperref}
\usepackage{graphicx}
\usepackage{verbatim}

\usepackage{amsmath,amssymb,amsthm}

\renewcommand\[{\begin{equation}}\renewcommand\]{\end{equation}}
\renewcommand\epsilon\varepsilon
\renewcommand\phi\varphi
\renewcommand\geq\geqslant
\renewcommand\leq\leqslant
\renewcommand\ln\log

\newcommand\RR{\mathbb{R}}

\newcommand\ab\allowbreak

\theoremstyle{definition}

\theoremstyle{plain}

\newtheorem{thm}{Theorem}[section]
\newtheorem{defn}[thm]{Definition}
\newtheorem{prop}[thm]{Proposition}
\newtheorem{lem}[thm]{Lemma}
\newtheorem{rque}{Remark} [section]
\newtheorem{cor}[thm]{Corollary}

\newtheorem{rmq}{Remark}[section]

\DeclareMathOperator{\Ent}{Ent}

\newcommand{\R}{\mathbb{R}}
\newcommand{\N}{\mathbb{N}}

\newcommand{\E}{\mathbb{E}}
\newcommand{\tQ}{\widetilde{Q}}
\newcommand{\tW}{\widetilde{W}}
\newcommand{\cX}{\mathcal{X}}
\newcommand{\cP}{\mathcal{P}}
\newcommand{\I}{\mathcal{I}}
\begin{document}

\title{Curvature and transport inequalities for Markov chains in discrete spaces}
\date{\today}

\author{Max Fathi
\thanks{University of California, Berkeley, United States, mfathi@phare.normalesup.org. Supported in part by NSF FRG grant DMS-1361185.}
 and Yan Shu
\thanks{Universit\'e Paris-Ouest Nanterre, France, yshu@u-paris10.fr}
}
\maketitle

\begin{abstract}
We study various transport-information inequalities under three different notions of Ricci curvature in the discrete setting: the curvature-dimension condition of Bakry and \'Emery \cite{BE85}, the exponential curvature-dimension condition of Bauer \textit{et al.} \cite{BHLLMY} and the coarse Ricci curvature of Ollivier \cite{Oll2009}. We prove that under a curvature-dimension condition or coarse Ricci curvature condition, an $L_1$ transport-information inequality holds; while under an exponential curvature-dimension condition, some weak-transport information inequalities hold. As an application, we establish a Bonnet-Meyer's theorem under the curvature-dimension condition CD$(\kappa,\infty)$ of Bakry and \'Emery \cite{BE85}.
\end{abstract}

\section{Introduction}

In the analysis of the geometry of Riemannian manifolds, Ricci curvature plays an important role. In particular, Ricci curvature lower bounds immediately yield powerful functional inequalities, such as the logarithmic Sobolev inequality, which in turn implies transport-entropy and transport-information inequalities.
 Each of these inequalities has its own interest and has various applications, such as concentration bounds and estimates on the speed of convergence to equilibrium for Markov chains. We refer the reader to \cite{BGL14,GLWW2,Vi08} for more about the links between curvature and functional inequalities, and to \cite{An00,Led} for applications of functional inequalities.

However, when the space we consider is a graph, those theories are not as clear as in the continuous settings. The first question one would want to answer is how to define Ricci curvature lower bounds in discrete settings. The natural approach would be to define it as a discrete analogue of a definition valid in the continuous setting. There are several equivalent definitions one can try to use (see \cite{AGS15} for those definitions in the continuous settings and for the equivalences between them). However, in discrete spaces, we lose the chain rule, and these definitions are no longer equivalent.

Several notions of curvature have been proposed in the last few years. Here we shall consider three of them: the curvature-dimension condition of Bakry and \'Emery \cite{BE85}, the exponential curvature-dimension condition of Bauer \textit{et al.} \cite{BHLLMY} and the coarse Ricci curvature of Ollivier \cite{Oll2009}. Other notions that have been developed (and which we shall not discuss further here) include the entropic Ricci curvature defined by Erbar and Maas in \cite{EM12} and Mielke in \cite{Mi13}, which is based on the Lott-Sturm-Villani definition of curvature \cite{LV09,St06}, geodesic convexity along interpolations in \cite{GRST2014} and \cite{L2014}, rough curvature bounds in \cite{BS09}. It is still an open problem to compare these various notions of curvature. We refer readers to the forthcoming survey \cite{CGMPRSST} for a more general introduction.

The aim of this work is to obtain functional inequalities under the above three notions of curvature conditions and give some applications.

Let us begin with setting the framework of Markov chains on discrete spaces:

\subsection*{Markov chain on graphs and curvature condition}
Let $\cX$ be a finite (or countably infinite) discrete space and $K$ be an irreducible Markov kernel on $\cX$. Assume that for any $x\in \cX$, we have
\begin{equation} \label{assumpt_markov}
\underset{y}{\sum} \hspace{1mm} K(x,y) = 1.
\end{equation}
This condition is a normalization of the time scale, enforcing that jump attempts occur at rate $1$. We also define $J(x):=1-K(x,x)$ and $J:=\sup_{x\in \cX}J(x)$. $J$ is a measure of the laziness of the chain, estimating how often jump attempts end with the particle not moving. Since we assume the kernel is irreducible, $0 < J \leq 1$.
\newline
We shall always assume there exists a reversible invariant probability measure $\pi$, satisfying the detailed balance relation
$$K(x,y)\pi(x) = K(y,x)\pi(y) \hspace{5mm} \forall x, y \in \cX.$$
We denote by $L$ the generator of the continuous-time Markov chain associated to the kernel $K$, which is given by
$$Lf(x) = \underset{y}{\sum} (f(y) - f(x))K(x,y).$$
Let $P_t = e^{tL}$ be the associated semigroup, acting on functions, and $P_t^*$ its adjoint, acting on measures. We also define the $\Gamma$ operator, given by
$$\Gamma(f,g)(x) := \frac{1}{2}\underset{y}{\sum} \hspace{1mm}(f(y)-f(x))(g(y) - g(x))K(x,y)$$
and write $\Gamma(f) := \Gamma(f,f)$.

With this $\Gamma$ operator, we are able to introduce the Bakry-\'Emery curvature condition CD$(\kappa$, $\infty)$:

\begin{defn}
We define the iterated $\Gamma$ operator $\Gamma_2$ as
$$\Gamma_2(f) = \frac{1}{2}L\Gamma(f) - \Gamma(f, Lf).$$
We say that the curvature condition CD$(\kappa$, $\infty)$ is satisfied if, for all functions $f$, we have
$$\Gamma_2(f) \geq \kappa \Gamma(f).$$
\end{defn}

Since we shall deal with three different types of Ricci curvature lower bounds, in order to avoid confusions, we always denote $\kappa$ for the Bakry-\'Emery curvature condition, $\kappa_e$ for the exponential curvature dimension condition and $\kappa_c$ for a lower bound on the Coarse Ricci curvature. Throughout the paper, $\kappa,\kappa_e, \kappa_c$ will always be positive numbers.

The $\Gamma_2$ operator and the curvature condition were first introduced in \cite{BE85}, and used to prove functional inequalities, such as logarithmic Sobolev inequalities and Poincar\'e inequalities, for measures on Riemannian spaces satisfying CD($\kappa$, $\infty$) for $\kappa > 0$. In the Riemannian setting, the $\Gamma_2$ operator involves the Ricci tensor of the manifold, and the condition CD$(\kappa$, $\infty)$ is equivalent to asking for lower bounds on the Ricci curvature, and more generally to the Lott-Sturm-Villani definition of lower bounds on Ricci curvature (see \cite{LV09} and \cite{St06} for the definition, and \cite{AGS15} for the equivalence between the two notions). Hence CD$(\kappa$, $\infty)$ can be used as a \emph{definition} of lower bounds on the Ricci curvature for nonsmooth spaces, and even discrete spaces. This was the starting point of a very fruitful direction of research on the links between curvature and functional inequalities. In most cases, the focus was on the continuous setting, and the operator $L$ was assumed to be a diffusion operator.

In the discrete setting, this curvature condition was first studied by Schmuckenshlager in \cite{Sc} and then by S.-T. Yau and his collaborators in \cite{BHLLMY, BLLY,CLY,LY10}. It has also been used in \cite{KKRT14}, where a discrete version of Buser's inequality was obtained, as well as curvature bounds for various graphs, such as abelian Cayley graphs and slices of the hypercube. Note that most of these works are set in the framework of graphs rather than Markov chains, which generally makes our definitions and theirs differ by a normalization constant, since we enforce the condition \eqref{assumpt_markov}.

As we have mentioned, the main differences between the continuous and discrete settings is the when the operator $L$ is not a diffusion operator, we lose the chain rule. This leads to additional difficulties, and some results, such as certain forms of the logarithmic Sobolev inequality, do not seem to hold anymore. On the other hand, one of the main difficulties in the continuous setting is to exhibit an algebra of smooth functions satisfying certain conditions, while this property immediately holds in the discrete setting.

The key chain rule used in the continuous setting is the identity
$$L(\Phi(f)) = \Phi'(f)Lf + \Phi''(f)\Gamma(f)$$
which characterizes diffusion operators in the continuous setting, and does not hold in discrete settings. However, a key observation of \cite{BHLLMY} is that when $\Phi(x) = \sqrt{x}$, the identity
$$2\sqrt{f}L\sqrt{f} = Lf - 2\Gamma(\sqrt{f})$$
holds, even in the discrete setting. This observation motivated the introduction of a modified version of the curvature-dimension condition, designed to exploit this identity:

\begin{defn}
We define the modified $\Gamma_2$ operator $\tilde{\Gamma}_2$ as
$$\tilde{\Gamma}_2(f,f) := \Gamma_2(f) - \Gamma \left(f, \frac{\Gamma(f)}{f}\right).$$
We say that the exponential curvature condition CDE'($\kappa_e$, $\infty$) is satisfied if, for all nonnegative functions $f$ and all $x \in \mathcal{X}$, we have
$$\tilde{\Gamma}_2(f)(x) \geq \kappa_e \Gamma(f)(x).$$
\end{defn}

\begin{rmq}
We use the notation CDE'($\kappa_e$, $\infty$) to agree with the notations of \cite{BHLLMY}, where they also consider the case when the condition is only satisfied at points $x$ where $L f(x) < 0$.

In \cite{Mu15}, it is shown that CDE'$(\kappa_e$, $\infty)$ implies CD$(\kappa$, $\infty)$ with $\kappa=\kappa_e$. When the operator $L$ is a diffusion, the conditions CDE'$(\kappa_e$, $\infty)$ and CD$(\kappa$, $\infty)$ are equivalent.
\end{rmq}

Under this notion of curvature, Bauer et al. prove in \cite{BHLLMY} various Li-Yau inequalities on graphs, and then deduce heat kernel estimates and a Buser inequality for graphs. In \cite{BLLY}, it was shown that the CDE'($\kappa_e$, $\infty$) condition tensorizes, and that the associated heat kernel satisfies some Gaussian bounds.

The third notion of curvature we shall now introduce is the coarse Ricci curvature. In order to define it, we first need to introduce Wasserstein distances.

Let $d$ be a distance on $\cX$. The $L^p$ Wasserstein distance is defined as following:
\begin{defn}[$L^p$-Wasserstein distances]
Let $p \geq 1$. The $L^p$-Wasserstein distance $W_p$ between two probability measures $\mu$ and $\nu$ on a metric space $(\cX, d)$ is defined as
$$W_p(\mu, \nu) := \left(\underset{\pi}{\inf} \hspace{1mm} \int{d(x,y)^p\pi(dxdy)} \right)^{1/p}$$
where the infimum runs over all couplings $\pi$ of $\mu$ and $\nu$.
\end{defn}

Finally, we recall the definition of Coarse Ricci curvature, which has been introduced by Ollivier in \cite{Oll2009} for discrete-time Markov chains. Since we shall work in continuous time, we shall give the appropriate variant, introduced in \cite{J07}. Previous works considering contraction rates in transport distance include \cite{D70,Sa05,Ol2011}. Applications to error estimates for Markov Chain Monte Carlo were studied in \cite{JO2010}. The continuous-time version we use here was introduced in \cite{J07}. The particular case of curvature on graphs has been studied in \cite{JL14}.

\begin{defn}[Coarse Ricci curvature]
The coarse Ricci curvature of the Markov chain is said to be bounded from below by $\kappa_c$ if, for all probability measures $\mu$ and $\nu$ and any time $t \geq 0$, we have
$$W_1(P_t^*\mu, P_t^*\nu) \leq \exp(-\kappa_c t)W_1(\mu, \nu),$$
i.e. if it is a contraction in $W_1$ distance, with rate $\kappa_c$.
\end{defn}

Note that unlike the CD($\kappa$, $\infty$) condition, this property does not only depend on the Markov chain, but also on the choice of the distance $d$.

In this work, there are two distances we shall be interested in. In the rest of the paper, we write $d_g$ for the \emph{graph distance} associated to the Markov kernel. If we consider $\cX$ as the set of vertices of a graph, with edges between all pairs of vertices $(x,y)$ such that $K(x,y) > 0$, $d_g$ shall be the usual graph distance. More formally, it is defined as
$$d_g(x,y) := \inf \{n \in \N; \exists x_0,..,x_n | x_0 = x, \hspace{1mm} x_n = y, \hspace{1mm} K(x_i, x_{i+1}) > 0 \hspace{2mm} \forall 0 \leq i \leq n-1 \}.$$

In section \ref{sect_CD}, we shall also consider the distance $d_{\Gamma}$, defined by
$$d_{\Gamma}(x,y) = \underset{f; \Gamma(f) \leq 1}{\sup} \hspace{1mm} f(x) - f(y).$$

The main property that makes $d_{\Gamma}$ interesting is that the 1-lipschitz functions for $d_{\Gamma}$ are automatically characterized as the functions $f$ such that $\Gamma(f) \leq 1$. In this case, we denote $W_{p,d_\Gamma}$ instead of $W_p$ for the Wasserstein distance of space $(\cX,d_\Gamma)$. One reason to consider this distance is that $d_\Gamma$ is the exact analog of the classical situation for continuous metric space, where $\Gamma(f) = |\nabla f|^2$. On the other hand, observe that $\Gamma(d(x,\cdot)) \leq \frac{J(x)}{2}$ holds for all $x$. It follows that
 \begin{equation}\label{d_and_d_gamma}
d_g(x,.) \leq \sqrt{\frac{J(x)}{2}}d_{\Gamma}(x,.).
\end{equation}
Thus for all $x,y\in \cX$,
\begin{equation}\label{d_and_d_gamma2}
d_g(x,y) \leq \min\{\sqrt{\frac{J(x)}{2}},\sqrt{\frac{J(y)}{2}}\}d_{\Gamma}(x,y).
\end{equation}
Therefore the estimates on $d_{\Gamma}$ are stronger than estimates on $d_g$ with a constant $\sqrt{\frac{J}{2}}$ which is bounded by $\sqrt{\frac{1}{2}}$. Of course, this also means that functional inequalities involving $d$ shall be easier to obtain, at least in some situations.

\subsection*{Functional inequalities}
Now we turn to functional inequalities on graphs:

\begin{defn}(Fisher information)
Let $f$ be a nonnegative function defined on $\cX$. Define the Fisher information $I_{\pi}$ of $f$ with respect to $\pi$ as
$$\I_{\pi}(f) := 4\int \Gamma(\sqrt{f})d\pi= 2\sum_{x\in cX}\sum_{y\in \cX} (\sqrt{f(y)}-\sqrt{f(x)})^2K(x,y)\pi(x).$$
The factor 4 in this definition comes from the analogy with the continuous setting, where $$4\int{|\nabla \sqrt{f}|^2d\pi} = \int{|\nabla \log f|^2fd\mu}.$$
In the continuous setting, the Fisher information can be written as $\int{\nabla \log f \cdot \nabla f d\pi}$, so we can define a modified Fisher information as
\begin{equation}
\tilde{\I}_{\pi}(f) := \int{\Gamma(f, \log f)d\pi},
\end{equation}
which corresponds to the entropy production functional of the Markov chain.

There is a third way to rewrite the Fisher information for the continuous settings as $\int{\frac{|\nabla f|^2}{f}d\mu}$, and one can also define another modified Fisher information as
$$\overline\I_{\pi}(f) :=\int \frac{\Gamma(f)}{f}d\pi.$$
\end{defn}

Of course, there are many other ways to re-write the Fisher information in the continuous setting, each leading to a different definition in the discrete setting. We only stated here the three versions we shall use in this work.

In the discrete setting, $\I_\pi(f)$, $\tilde{\I}_\pi(f)$ and $\overline\I_{\pi}(f)$ are not equal in general. It is easy to see that $\I_\pi(f) \leq \tilde{\I}_\pi(f)$ and $\I_\pi(f) \leq \overline\I_{\pi}(f)$
If $f$ is the density function of a probability measure $\nu$ with respect to $\pi$, since $(\sqrt{f(y)}-\sqrt{f(x)})^2\leq f(x)+f(y)$, and since $\pi$ is reversible, one can deduce that
$$\I_{\pi}(f)\leq 2\sum_{x\in \cX}\sum_{y\in \cX} \left(f(x)+f(y)\right)K(x,y)\pi(x)\leq 4J.$$
Here we can see that in discrete settings, the Fisher information is in fact bounded from above, which is not true in continuous settings.

Let us recall the definition of the relative entropy $\Ent_{\pi}$ as well:
\begin{defn}(Relative entropy)
Assuming that $f$ is a nonnegative function on $\cX$, we define the relative entropy $f$ with respect to $\pi$ as following:
$$\Ent_{\pi}(f) := \sum_\cX f(x)\log f(x) \pi(x)- \sum_\cX f(x)\pi(x)\log \left(\sum_\cX f(x)\pi(x)\right).$$
Note that when $f$ is a probability density with respect to $\pi$, the second term takes value $0$.
\end{defn}

%

\begin{defn}\label{deffunctionalinequality}
Let $\pi$ be a probability measure on $\cX$ and $p\geq 1$. We say that $\pi$ satisfies

(i) the logarithmic Sobolev inequality with constant $C$, if for all nonnegative functions $f$, we have
$$\Ent_{\pi}(f) \leq \frac{1}{2C}\I_{\pi}(f) \qquad (LSI(C));$$

(ii) the modified logarithmic Sobolev inequality with constant $C$, which we shall write $mLSI(C)$, if for all nonnegative functions $f$, we have
$$\Ent_{\pi}(f) \leq \frac{1}{2C}\int{\Gamma(f, \log f)d\pi} \qquad (mLSI(C));$$

(iii) the transport-entropy inequality $T_p(C)$ if for all probability measures $\nu = f\pi$, we have
$$W_p(\nu, \mu)^2 \leq \frac{2}{C}\Ent_{\pi}(f)\qquad (T_p(C));$$

(iv) the transport-information inequality $T_pI(C)$ if for all probability measures $\nu = f\pi$, we have
$$W_p(\nu, \mu)^2 \leq \frac{1}{C^2}\I_{\pi}(f)\qquad (T_pI(C)).$$
\end{defn}

In the continuous setting, $mLSI$ and $LSI$ are the same inequality, but in the discrete setting they correspond to distinct properties of the Markov chain, namely hypercontractivity for $LSI$ and exponential convergence to equilibrium in relative entropy for $mLSI$. In general, $LSI$ implies $mLSI$, but the converse is \emph{not} true. We refer to \cite{BT} for more on the difference between the two inequalities.

In the discrete setting, when $p=1$, the following relations between the inequalities hold, in the same way as in the continuous setting:
$$\text{LSI}(C) \Rightarrow T_1I(C) \Rightarrow T_1(C).$$
The $T_1$ inequality is equivalent to Gaussian concentration for $\pi$ (see for example \cite{Led} and the next section), but not dimension-free concentration, and is therefore strictly weaker than $T_2$. When $p=1$, the transport-information inequality is equivalent to Gaussian concentration for the occupation measure of the Markov chain (see \cite{GLWY2009}), and is therefore useful to get a priori bounds on the statistical error for Markov Chain Monte Carlo estimation of averages.

One of the most interesting cases in the continuous setting is the transport-entropy inequality when $p=2$, which is also called the Talagrand inequality, which was introduced in \cite{Tal}. It is equivalent to dimension-free Gaussian concentration for $\pi$.

One can show the following relationships:
$$CD(\kappa,\infty)\Rightarrow LSI(\kappa) \Rightarrow T_2I(\kappa)\Rightarrow T_2(\kappa).$$
We refer to \cite{OV} and \cite{GLWW2} for the proofs of these implications. With those inequalities in hand, one can prove some dimension-free concentration results on a metric-measure space (see \cite{Go09}).

However, the results for $p=2$ fail to be true in discrete settings, as we will see in next section. When $\cX$ is a graph, $\pi$ never satisfies $T_2$, unless it is a Dirac measure (see for example \cite{GRS14}, or Section \ref{prelim}). To recover a discrete version of $T_2$, we therefore have to redefine the transport cost. Erbar and  Maas recovered some of those functional inequality results with the notion of entropic Ricci curvature on graphs, and we refer the reader to \cite{EM12,Maas11} for more details. Another way to deal with it is to take the weak transport cost introduced by Marton \cite{Mar96}:

\begin{defn}
Let $(\cX,d)$ be a polish space and $\mu,\nu$ two probabilities measures on $\cX$, define
$$\widetilde{\mathcal{T}}_2(\nu|\mu):=\inf_{\pi\in\Pi(\mu,\nu)}\left\{\int \left(\int d(x,y)p_x(dy)\right)^2\mu(dx)\right\}.$$
Where $\Pi(\mu,\nu)$ is the set of all couplings $\pi$ whose first marginal is $\mu$ and second marginal is $\nu$, $p_x$ is the probability kernel such that $\pi(dxdy)=p_x(dy)\mu(dx)$.
Using probabilistic notations, on has
$$\widetilde{\mathcal{T}}_2(\nu|\mu)=\inf_{X\sim \mu,Y\sim \nu} \E((\E(d(X,Y)|X)^2)).$$
\end{defn}
Note that the weak transport cost could be also seen as a weak Wasserstein-like distance. In order to agree with the notations of Wassertein distance, we note $\tW_2(\nu|\mu)^2:= \widetilde{\mathcal{T}}_2(\nu|\mu)$. However, it is \emph{not} a distance, since it is not symmetric. Note that by Jensen's inequality, both $\tW_2(\nu|\mu)$ and $\tW_2(\mu,\nu)$ are larger than $W_1(\mu,\nu)$.

\begin{defn}
Adapting the settings of Definition \ref{deffunctionalinequality}, we say that $\pi$ satisfies

(v) the weak transport-entropy inequality $\widetilde{T}^+_2(C)$ if for all probability measures $\nu = f\pi$, we have
$$\tW_2(f\pi|\pi)^2 \leq \frac{2}{C}\Ent_{\pi}(f);$$
(vi) the weak transport-entropy inequality $\widetilde{T}^-_2(C)$ if for all probability measures $\nu = f\pi$, we have
$$\tW_2(\pi|f\pi)^2 \leq \frac{2}{C}\Ent_{\pi}(f);$$
(vii) the weak transport-information inequality $\widetilde{T}^+I_2(C)$ if for all probability measures $\nu = f\pi$, we have
$$\tW_2(f\pi| \pi)^2 \leq \frac{1}{C^2}\I_{\pi}(f).$$
(viii)  the weak transport-information inequality $\widetilde{T}^-I_2(C)$ if for all probability measures $\nu = f\pi$, we have
$$\tW_2(\pi|f\pi)^2 \leq \frac{1}{C^2}\I_{\pi}(f).$$
\end{defn}

Here we only consider the case when the cost function is quadratic, for more general result about weak transport inequalities, we refer to \cite{GRST15} and \cite{Sh15}.

Our main results are the following theorems:

\begin{thm}\label{CDW1I}
Let $(\cX,d_\Gamma)$ be a connected graph equipped with $\Gamma$-distance $d_\Gamma$. Let $K$ be an irreducible Markov kernel on $\cX$ and $\pi$ the reversible invariant probability measure associated to $K$. Assume that $CD(\kappa,\infty)$ holds with $\kappa>0$. Then $\pi$ satisfies the transport-information inequality $T_1I$ with constant $\kappa$. More precisely, for all probability measure $\nu:= f\pi$ on $\cX$, it holds
$$W_{1,d_\Gamma}(f\pi,\pi)^2\leq \frac{2}{\kappa^2}\I_\pi(f).$$
\end{thm}

With such a result in hand, we can then follow the work by Guillin, Leonard, Wang and Wu \cite{GLWW2}, to prove a transport-entropy inequality $T_1$ holds, so that the Gaussian concentration property follows as well. Another application is that after a simple computation, one can obtain the following Bonnet-Meyer type theorem:

\begin{cor}\label{dia_estimate}
Assume that $CD(\kappa,\infty)$ holds, then
$$d_g(x,y)\kappa \leq 2\min\{\sqrt{J(x)},\sqrt{J(y)}\}\left(\sqrt{J(x)}+\sqrt{J(y)}\right).$$
\end{cor}
Recall that under coarse Ricci curvature condition, Ollivier has proved in \cite{Oll2009} the same type inequality:  $\kappa_c d(x,y)\leq J(x)+J(y)$. Here we get a diameter estimate with a better order but we lose a constant 2 order under the condition $CD(\kappa,\infty)$.

Now if we make the stronger assumption $CDE'(\kappa_e,\infty)$, we get a stronger inequality:
\begin{thm}\label{thmW2I}
Let $(\cX,d)$ be a connected graph equipped with graph distance $d$. Let $K$ be a irreducible Markov kernel on $\cX$ and $\pi$ the reversible invariant probability measure associated to $K$. Assume that $CDE'(\kappa_e,\infty)$ holds with $\kappa_e>0$. Then $\pi$ satisfies the transport-information inequality $T_2^+I$ with constant $\kappa_e/\sqrt{2}$. More precisely, for all probability measure $\nu:= f\pi$ on $\cX$, it holds
$$\tW_2(f\pi|\pi)^2\leq \frac{2J}{\kappa_e^2}\I_\pi(f) \leq \frac{2}{\kappa_e^2}\I_\pi(f).$$
\end{thm}
Again, following the ideas of \cite{GLWW2}, one can prove a weak-transport entropy inequality $\widetilde {T}_2^+H$. On the other hand, sine the weak-transport cost is stronger than the $L^1$-Wasserstein distance, it yields immediately $T_1I$ holds, which implies $T_1H$ and concentration results.

Under coarse Ricci curvature condition, the inequality $T_1I(\kappa_c)$ holds
\begin{thm}\label{thmWIollivier}
Let $\cX,\pi,K$ define as before.
If the global coarse Ricci curvature is bounded from below by $\kappa_c>0$, then the following transport inequality holds for all density function $f$:
$$W_{1}(f\pi,\pi)^2\leq \frac{1}{\kappa_c^2}\I_\pi(f)\left(J-\frac{1}{8}\I_\pi(f)\right) \leq \frac{1}{\kappa_c^2}\I_\pi(f).$$
\end{thm}
As a corollary, this last result implies a $T_1$ inequality for such Markov chain, which has been previously obtained by Eldan, Lehec and Lee \cite{ELL}.

The paper is organized as follows: in first section we will explain why $T_2H$ can not hold in general, then establish connection of results in \cite{Sh15} and the Fisher information on graph settings. Section 2 gives a few preliminary results about Hamilton-Jacobi equations on graphs. In the third, fourth and fifth sections we will discuss functional inequalities under $CD(\kappa,\infty)$, $CDE'(\kappa_e,\infty)$ and coarse Ricci curvature $\kappa_c$ respectively. In the last section we will show some applications, such as how a transport-information inequality implies a transport-entropy inequality, concentration results, a discrete analogue of the Bonnet-Meyer theorem, and a study of the example of the discrete hypercube.
%
%
%

\section{Preliminary}
\label{prelim}
In this section, we present some general results in the discrete setting, without assuming any curvature condition. Our main concern is to present the Hamilton-Jacobi equations on graphs introduced in \cite{GRST15,Sh15} and their relation with weak transport costs.

As we mentioned in the introduction, in the discrete setting, when $\pi$ is not a Dirac mass, the inequality $T_2(\kappa)$ cannot hold true, for any $\kappa>0$. To our knowledge, this was first proved in \cite{GRS14}. We give here a different proof, as a consequence of a more general result:
\begin{lem}
Let $(\cX,d)$ be a metric space and $\mu$ a probability measure defined on $\cX$. Assume that there exist $C_1, C_2\subset \cX$ such that

$(i)$ $\inf_{x\in C_1,y\in C_2}d(x,y)>0$,

$(ii)$ $\mathrm{supp} (\mu) \subset C_1\cup C_2$,

$(iii)$  $\mu(C_1)>0$, $\mu(C_2)>0$.

Then $\mu$ does not satisfies $T_2(\kappa)$ for any $\kappa>0$.
\end{lem}

\begin{proof}
For $h< \min\{\mu(C_1),\mu(C_2)\}$, define

 $$\nu_h(dx):=\left\{ \begin{array}{ll}
  \mu(dx)(1+\frac{h}{\mu(C_1)}), & x\in C_1 \\
  \mu(dx)(1-\frac{h}{\mu(C_2)}), & x\in C_2
  \end{array}
  \right.$$
Let $d:=\inf_{x\in C_1,y\in C_2}d(x,y)>0$. Then we have $W_2(\mu,\nu)^2 \geq d^2h$, and the entropy is
$$(\mu(C_1)+h)\log(1+h/\mu(C_1))+(\mu(C_2)-h)\log(1-h/\mu(C_2)).$$
When $h$ goes to $0$, the entropy is $\mathcal{O}(h^2)$. The conclusion follows since $W_2^2$ have order $\mathcal{O}(h)$.
\end{proof}
Thus, if $\pi$ satisfies $T_2$ on graphs, it means that $\pi$ is a Dirac mass.

In this section, we shall describe the links between weak transport inequalities and the Hamilton Jacobi operator that was introduced in \cite{GRST15} and studied in \cite{Sh15}.

Following \cite{GRST15}, the weak transport cost $\tW_2^2$ between two probability measures $\mu$ and $\nu$ satisfies the following duality formula:
\begin{equation}\label{duality}
\tW_2(\nu|\mu)^2=\sup_{g\in \mathcal{C}_c^b} \left\{\int \tQ_1g d\nu- \int g d\mu\right\}
\end{equation}
where the infimum-convolution operator is defined as
$$\tQ_t\phi(x):=\inf_{p\in \cP(\cX)}\left\{\int g(y) p(dy) + \frac{1}{t}\left(\int d(x,y)p(dy)\right)^2\right\},$$

Later, the second author remarked (see \cite{Sh15}) that the operator $\tQ_t$ satisfies a discrete version of the Hamilton-Jacobi equation: for all $t>0$
\begin{equation}\label{HJ}
\frac{\partial}{\partial t}\tQ_t g + \frac{1}{4}|\widetilde{\nabla} \tQ_tg|^2 \leq 0
\end{equation}
where $|\widetilde{\nabla}g|(x):= \sup_{y\in \cX} \frac{[g(y)-g(x)]_-}{d(x,y)}$. We refer to \cite{BGL14, Vi08} for information about Hamilton-Jacobi equations in the continuous setting and their link with functional inequalities.

Now let $\alpha\in \mathcal{C}_1(\RR^+)$, according to \eqref{HJ}, one can easily check that for all $t>0$, it holds
\begin{equation}\label{HJalpha}
|\alpha'(t)\frac{\partial}{\partial t}\tQ_{\alpha(t)} g| \geq \frac{1}{4} |\widetilde{\nabla} \tQ_{\alpha(t)}g|^2.
\end{equation}

The evolution with respect to time is controlled by this special "gradient". We refer readers to \cite{Sh15} for properties of $\tQ$ and $\widetilde \nabla$. Here we shall develop some more:

\begin{prop}(Convexity)\label{convexity of Q}
Let $g$ be a function defined on $\cX$, then for all $x\in \cX$, the function $t\mapsto \tQ_tg(x)$ is convex.
\end{prop}
\begin{proof}
Fix $x\in \cX$, define $G(t):=\tQ_tg(x)$
Observe that for any $\lambda\in [0,1]$, and $p_1, p_2\in \mathcal{P}(\cX)$, setting $p:=\lambda p_1+(1-\lambda) p_2\in \mathcal{P}(\cX)$ and applying the Cauchy-Schwarz inequality, it holds for all $t,s>0$:
\begin{multline}
\left(\int d(x,z)p(dz)\right)^2=\left(\lambda \int d(x,z)p_1(dz)+(1-\lambda)\int d(x,z)p_2(dz)\right)^2\\
\leq (\lambda t+(1-\lambda)s)\left(\frac{\lambda (\int d(x,z)p_1(dz))^2}{t}+\frac{(1-\lambda) (\int d(x,z)p_2(dz))^2}{s}\right)
\end{multline}
As a consequence, we get
\begin{multline}
 \lambda\left(\int g(z)p_1(dz) +\frac{1}{t}(\int d(x,z)p(dz))^2\right)+(1-\lambda)\left(\int g(z)p_2(dz) +\frac{1}{s}\left(\int d(x,z)p_2(dz)\right)^2\right)\\
 \geq \int g(z)p(dz) +\frac{1}{(\lambda t+(1-\lambda)s)}\left(\int d(x,z)p(dz)\right)^2
 \geq G(\lambda t+(1-\lambda) s)
\end{multline}
Taking the infimum over all $p_1,p_2\in \mathcal{P}(\cX)$ on left hand side of the inequality, the conclusion follows.
\end{proof}

The following lemma is a technical result connecting the gradient $\widetilde \nabla$ and $\Gamma$-operator.
\begin{lem}\label{Estimation of Gamma}
Let $\pi$ be the reversible probability measure for the Markov kernel $k$. For any bounded function $f$ and $g$ on $\cX$, the following inequalities hold:
\begin{description}
\item $(i)$ $\int \Gamma(f,g)d\pi\leq \sqrt{2}J\int |\widetilde{\nabla}g||\widetilde{\nabla}f|d\pi$,

\item $(ii)$ $|\int \Gamma(f,g)d\pi|\leq \sqrt{2J}\int |\widetilde{\nabla}g|\sqrt{\Gamma(f)}d\pi$.
\end{description}
Moreover, if we suppose that $f$ is non negative, then we have
\begin{description}
\item $(iii)$ $\int \Gamma(f,g)d\pi\leq 2\sqrt{2J}\int |\widetilde{\nabla}g|\sqrt{f\Gamma(\sqrt{f})}d\pi$,
\item $(iv)$  $\int \Gamma(\sqrt{f})d\pi \leq \frac{J}{4}\int |\widetilde{\nabla}\log f|^2 f d\pi.$
\end{description}
\end{lem}
\begin{proof}
The proofs of these four inequalities all follow similar arguments.
Denote the positive part and negative part of a function $u$ as $u_+$ and $u_-$ respectively.

$(i)$:
Using the relation $(uv)_+\leq u_+v_++u_-v_-$, we have
\begin{align*}
\int \Gamma(f,g)_+d\pi
&=
\frac{1}{2}\sum_x\left[\sum_{y\sim x}(f(y)-f(x))(g(y)-g(x))\right]_+K(x,y)\pi(x)\\
&\leq
\frac{1}{2}\sum_x\sum_{y\sim x}(g(y)-g(x))_+(f(y)-f(x))_+K(x,y)\pi(x)\\
&+\frac{1}{2}\sum_x\sum_{y\sim x}(g(y)-g(x))_-(f(y)-f(x))_-K(x,y)\pi(x)
\end{align*}
Now by reversibility of the measure $\pi$, it holds
\begin{align*}
&\ \ \sum_x\sum_{x}(g(y)-g(x))_+(f(y)-f(x))_+K(x,y)\pi(x)\\
&=\sum_x\sum_{y\sim x}(g(y)-g(x))_-(f(y)-f(x))_-K(x,y)\pi(x)\\
&\leq \sum_x |\widetilde{\nabla}g|(x) \sum_{y\sim x}(f(x)-f(y))_-K(x,y)\pi(x),
\end{align*}
Where the latter inequality follows from $|\widetilde{\nabla}g|(x)\geq (g(y)-g(x))_-$ for all $y\sim x$.
Therefore, we get
\begin{equation}\label{eq9}
\int \Gamma(f,g)_+d\pi
\leq
\sum_x |\widetilde{\nabla}g| \sum_{y\sim x}(f(y)-f(x))_-K(x,y)\pi(x).
\end{equation}
In \eqref{eq9}, using $\Gamma(f,g)\leq \Gamma(f,g)_+$ and $\sum_{y\sim x}(f(y)-f(x))_-K(x,y)\leq |\widetilde \nabla f|(x)J(x)$, we get $(i)$.

$(ii)$:
 By the Cauchy-Schwarz inequality, it holds
\begin{equation}\label{eq10}
\left(\sum_{y\sim x} (f(y)-f(x))_-K(x,y)\right)^2\leq J(x)\sum_{y\sim x} (f(y)-f(x))_-^2K(x,y)\leq 2J\Gamma(f)
\end{equation}
Combining \eqref{eq9} and \eqref{eq10} leads to
\begin{equation} \label{eq7}
\int \Gamma(f,g)_+d\pi \leq
\int |\widetilde{\nabla}g|\sqrt{2J\Gamma(f)}d\pi.
\end{equation}
Following a similar argument, we have
\begin{equation}\label{eq8}
\int \Gamma(f,g)_-d\pi\leq \int |\widetilde{\nabla}g|\sqrt{2J\Gamma(f)}d\pi.
\end{equation}
and $(ii)$ follows by \eqref{eq7}, \eqref{eq8} and the inequality
$$\left|\int \Gamma(f,g)d\pi\right|\leq \max\left\{\int \Gamma(f,g)_+d\pi,\int \Gamma(f,g)_-d\pi\right\}.$$

$(iii)$:
Since $f$ is nonnegative, $\sqrt{f}$ is well defined. Then it holds
\begin{align*}
\int \Gamma(f,g)d\pi
&=
\frac{1}{2}\sum_x\sum_{y\sim x}(f(y)-f(x))(g(y)-g(x))K(x,y)\pi(x)\\
&= \frac{1}{2}\sum_x\sum_{y\sim x}(g(y)-g(x))(\sqrt{f}(y)-\sqrt{f}(x))(\sqrt{f}(y)+\sqrt{f}(x))K(x,y)\pi(x)
\end{align*}
Now arguing as in $(i)$ and $(ii)$, by reversibility of $\pi$, we get
$$
\int \Gamma(f,g)d\pi =\sum_x\sum_{y\sim x}(g(y)-g(x))_-(\sqrt{f}(y)-\sqrt{f}(x))_-(\sqrt{f}(y)+\sqrt{f}(x))K(x,y)\pi(x).
$$
Notice that $(\sqrt{f}(y)-\sqrt{f}(x))_-(\sqrt{f}(y)+\sqrt{f}(x))\leq (\sqrt{f}(y)-\sqrt{f}(x))_-2\sqrt{f}(x)$, we have
\begin{align*}
\int \Gamma(f,g)d\pi
&\leq 2\sum_x\sum_{y\sim x}(g(y)-g(x))_-(\sqrt{f}(y)-\sqrt{f}(x))_-\sqrt{f}(x)K(x,y)\pi(x)\\
&\leq 2\sqrt{2J}\int |\widetilde{\nabla}g|\sqrt{f\Gamma(\sqrt{f})}d\pi
\end{align*}
where the last step we have used \eqref{eq10} with $u:=\sqrt{f}$.

$(iv)$:
If $f$ is the null function, there is nothing to say. Otherwise, if there exist $x,y\in \cX$ such that $f(x)=0,f(y)>0$, it is easy to see that $|\widetilde{\nabla} \log f(y)|^2f(y)\pi(y)=\infty$. So we only need to prove the case $f(x)>0$ for all $x\in\cX$.

Since $f$ is a positive function, one can rewrite $f=e^g$, it is enough to prove that
$$\int \Gamma(e^{g/2})d\pi \leq \frac{J}{4}\int |\widetilde{\nabla}g|^2e^gd\pi$$
holds for all function $g$.
In fact, by convexity of function $x\mapsto e^x$, we have for all $a>b$, $(a-b)e^a\geq e^a-e^b$. Thus
\begin{align*}
J\int |\widetilde{\nabla}g|^2e^g
&\geq \sum_{x\sim y;g(y)\leq g(x)}(g(x)-g(y))^2e^{g(x)}K(x,y)\pi(x)\\
&\geq 4\sum_{x\sim y;g(y)\leq g(x)}\left(e^{\frac{g(x)}{2}}-e^{\frac{g(y)}{2}}\right)^2K(x,y)\pi(x)\\
&=4\int \Gamma(e^{g/2})d\pi
\end{align*}
\end{proof}

\section{Transport inequalities for Markov chains satisfying CD$(\kappa, \infty)$}
\label{sect_CD}
In this section, we assume that the Markov chain satisfies the curvature condition CD$(\kappa,\infty)$ for some $\kappa>0$.
One of the main tools we shall use is the following sub-commutation relation between $\Gamma$ and the semigroup $P_t$, which was obtained in \cite{KKRT14}:
\begin{lem}  \label{prop_kkrt}
Assume that CD$(\kappa$, $\infty)$ holds. Then for any $f : \cX \longrightarrow \R$, we have
$$\Gamma(P_tf) \leq e^{-2\kappa t}P_t\Gamma(f).$$
\end{lem}

\begin{rque}
This property implies that if $f$ is $1$-Lipschitz for $d_{\Gamma}$, then $P_tf$ is $e^{-\kappa t}$-Lipschitz. Therefore, the condition CD$(\kappa,\infty)$ implies that the coarse Ricci curvature of the Markov chain, using the distance $d_{\Gamma}$, is bounded from below by $\kappa$.
\end{rque}

\subsection{$L^1$-transport inequalities}
Here we shall prove two inequalities connected to the $L^1$ Wasserstein distance  under CD$(\kappa,\infty)$.
\begin{proof}[Proof of Theorem \ref{CDW1I}]
The proof relies on the Kantorovitch-Rubinstein duality formula
$$W_{1,d_\Gamma}(\mu, \nu) = \underset{g \hspace{1mm} 1-lip}{\sup} \hspace{1mm} \int {gd\mu} - \int{g d\nu}.$$
Let $g$ be a $1$-Lipschitz function for $d_{\Gamma}$. This is equivalent to having $\Gamma(g) \leq 1$.

First, using the Cauchy-Schwartz inequality, it holds
\begin{align*}
-\int{\Gamma(P_tg, f)d\pi}
&=
-\frac{1}{2}\underset{x,y}{\sum} \hspace{1mm} (P_tg(y) - P_tg(x))(f(y) - f(x))K(x,y)\pi(x)\\
&=
\frac{1}{2}{\underset{x,y}{\sum} \hspace{1mm} |(P_tg(y) - P_tg(x))(\sqrt{f}(y) - \sqrt{f}(x))(\sqrt{f}(y) + \sqrt{f}(x))|K(x,y)\pi(x)}\\
&\leq\sum_x\pi(x)\Gamma(\sqrt{f})(x)^{\frac{1}{2}}\left(\sum_y(P_tg(y) - P_tg(x))^2(\sqrt{f}(y) + \sqrt{f}(x))^2K(x,y)\right)^{\frac{1}{2}},
\end{align*}
Now applying the Cauchy-Schwartz inequality again, the latter quantity is less than
$$\left(\int\Gamma (\sqrt{f})d\pi\right)^{\frac{1}{2}}\left(\underset{x,y}{\sum} \hspace{1mm} (P_tg(y) - P_tg(x))^2(\sqrt{f}(y) + \sqrt{f}(x))^2K(x,y)\pi(x)\right)^{\frac{1}{2}}.$$
Therefore, we have
\begin{align*}
-\int{\Gamma(P_tg, f)d\pi}
&\leq \left(\int \Gamma (\sqrt{f})d\pi\right)^{\frac{1}{2}}\left(\underset{x,y}{\sum} \hspace{1mm} (P_tg(y) - P_tg(x))^2(\sqrt{f}(y) + \sqrt{f}(x))^2K(x,y)\pi(x)\right)^{\frac{1}{2}}\\
&\leq \sqrt{2}\sqrt{\I_{\pi}(f)}\sqrt{\int \Gamma(P_tg)fd\pi},
\end{align*}
where the last step we have used the reversibility of the measure $\pi$ and the fact that $(\sqrt{f}(y) + \sqrt{f}(x))^2\leq 2(f(x)+f(y))$ for any nonnegative function $f$.

Therefore, according to Lemma \ref{prop_kkrt}, we have
\begin{align*}
\int{gd\pi} &- \int{g fd\pi} = \int_0^{+\infty}{\frac{d}{dt}\int{(P_tg)f d\pi}dt} \\
&= -\int_0^{+\infty}{\int{\Gamma(P_tg, f)d\pi}dt} \\
&\leq \int_0^{+\infty}{\sqrt{\I_{\pi}(f)}\sqrt{\int \Gamma(P_tg)fd\pi}dt} \\
&\leq \sqrt{2}\sqrt{\I_{\pi}(f)} \int_0^{+\infty}{e^{-\kappa t} \sqrt{\int P_t(\Gamma(g))f d\pi}dt} \\
&\leq \frac{\sqrt{2}}{\kappa}\sqrt{\I_{\pi}(f)}.
\end{align*}
The result immediately follows by taking the supremum over all $1$-Lipschitz functions $g$.
\end{proof}

Using \eqref{d_and_d_gamma}, we get the following corollary:
\begin{cor}\label{corW2ICD}
Assume that $CD(\kappa,\infty)$ holds with $\kappa>0$. Then $\pi$ satisfies the transport-information inequality $T_1I$ with constant $\kappa$. More precisely, for all probability measure $\nu:= f\pi$ on $\cX$, it holds
$$W_{1,d_g}(f\pi,\pi)^2\leq \frac{J}{\kappa^2}\I_\pi(f) \leq \frac{1}{\kappa^2}\I_\pi(f).$$
\end{cor}

Using, similar arguments, we can also prove the following Cheeger-type inequality:
\begin{prop} \label{lem_transport_gradient}
Assume that CD$(\kappa$, $\infty)$ holds. Then for any probability density $f$ with respect to $\pi$, we have
$$W_{1, d_{\Gamma}}(f\pi, \pi) \leq \frac{1}{\kappa}\int{\sqrt{\Gamma(f)}d\pi}.$$
\end{prop}
We call this a Cheeger-type inequality, by analogy with the classical Cheeger inequality
$$||f\pi - \pi||_{TV} \leq C\int{|\nabla f|d\pi}.$$
Here $\int{\sqrt{\Gamma(f)}d\pi}$ is an $L^1$ estimate on the gradient of $f$, while both $W_{1, d_{\Gamma}}(f\pi, \pi)$ and $||f\pi - \pi||_{TV}$ are distances of $L^1$ nature.

\begin{proof}
Once more, by Kantorovitch duality for $W_{1, d_\Gamma}$, and since $1$-Lipschitz functions are exactly the functions $g$ with $\Gamma(g) \leq 1$, we have
\begin{align*}
W_{1, d_{\Gamma}}(f\pi, \pi) &= \underset{g; \Gamma(g) \leq 1}{\sup} \hspace{1mm} \int{gfd\pi} - \int{gd\pi} \\
&= \underset{g; \Gamma(g) \leq 1}{\sup} \hspace{1mm}- \int_0^{+\infty}{\int{\Gamma(P_tg, f)d\pi}dt} \\
&\leq \underset{g; \Gamma(g) \leq 1}{\sup} \hspace{1mm} \int_0^{+\infty}{\int{\sqrt{\Gamma(P_tg)}\sqrt{\Gamma(f)}d\pi}dt} \\
&\leq \underset{g; \Gamma(g) \leq 1}{\sup} \hspace{1mm} \int_0^{+\infty}{e^{-\kappa t}\int{\sqrt{P_t\Gamma(g)}\sqrt{\Gamma(f)}d\pi}dt} \\
&\leq \frac{1}{\kappa}\int{\sqrt{\Gamma(f)}d\pi}
\end{align*}
\end{proof}
\begin{rque}
Again, by \eqref{d_and_d_gamma}, with assumption CD$(\kappa,\infty)$,we get
$$W_{1, d_g}(f\pi,\pi)\leq \sqrt{\frac{J}{2\kappa^2}}\int{\sqrt{\Gamma(f)}d\pi}.$$
\end{rque}

\subsection{$L^2$-transport inequalities}

Under condition CD$(\kappa,\infty)$, we have not been able to obtain a transport-entropy inequality involving a weak transport cost. However, we can still obtain a bound on $\tW_{2,d_g}(\pi|f\pi)^2$ with the Dirichlet energy.

\begin{prop}
Assume that CD$(\kappa$, $\infty)$ holds. Then for any probability density $f$ with respect to $\pi$, we have
$$\tW_{2,d_g}(\pi|f\pi)^2 \leq \frac{\sqrt{2}J}{\kappa^2}\int \Gamma(f)d\pi.$$
\end{prop}

Unlike the transport-information inequality, this inequality does not seem to imply a transport-entropy inequality, and does not seem to be directly related to concentration inequalities.

\begin{proof}
First, for any bounded continuous function $g$ on $\cX$, we have:
\begin{multline}\label{eq1}
\int \tQ g fd\pi-\int \tQ g d\pi=-\int_0^{+\infty}\frac{d}{dt}\int P_t(\tQ g)fd\pi dt\\
=-\int_0^{+\infty}\int{\Gamma(P_t(\tQ g), f)d\pi}dt
=\int_0^{+\infty}\int \sqrt{\Gamma(P_t(\tQ g))}\sqrt{\Gamma(f)}d\pi dt
\end{multline}
According to Lemma \ref{prop_kkrt}, CD$(\kappa,\infty)$ implies that $\Gamma(P_t(g))\leq e^{-2\kappa t}P_t(\Gamma(g))$ holds for all $g$.
Hence,
\begin{equation}\label{eq2}
\int_0^{+\infty}\int \sqrt{\Gamma(P_t(\tQ g))}\sqrt{\Gamma(f)}d\pi dt
\leq \int_0^{+\infty} e^{-\kappa t}\int \sqrt{P_t\Gamma((\tQ g))}\sqrt{\Gamma(f)}d\pi dt
\end{equation}
On the other hand, by Lemma \ref{convexity of Q} and \eqref{HJ}, it holds
\begin{multline}\label{eq3}
\int \tQ g d\pi-\int g d\pi=\int_0^1 \int \frac{d}{dt} \tQ_t g d\pi dt\\
\leq\int_0^1 \int \frac{d}{dt} \tQ_t g|_{t=1} d\pi dt
\leq -\frac{1}{4}\int |\widetilde{\nabla} \tQ g|^2 d\pi
\end{multline}
Now applying part $(i)$ of Proposition \ref{Estimation of Gamma}, we get
\begin{multline}\label{eq4}
  -\frac{1}{4}\int |\widetilde{\nabla} \tQ g|^2 d\pi \leq - \frac{1}{4\sqrt{2}J}\int \Gamma (\tQ g) d\pi \\
  =- \frac{1}{4\sqrt{2}J}\int P_t\Gamma (\tQ g) d\pi =\int_0^{+\infty}e^{-\kappa t}\int -\frac{\kappa}{4\sqrt{2}J} P_t\Gamma (\tQ g) d\pi dt
\end{multline}
Combining \eqref{eq1}, \eqref{eq2}, \eqref{eq3} and \eqref{eq4}, we have
\begin{align*}
\int \tQ g fd\pi-\int g d\pi
&=
\int \tQ g fd\pi-\int \tQ g d\pi+\int \tQ g d\pi-\int g d\pi
\\
&\leq
\int_0^{+\infty}e^{-\kappa t}\left(\int \sqrt{P_t\Gamma((\tQ g))}\sqrt{\Gamma(f)}-\frac{\kappa}{4\sqrt{2}J} P_t\Gamma (\tQ g) d\pi\right) dt\\
&\leq
\int_0^{+\infty}e^{-\kappa t}\int \frac{\sqrt{2}J}{\kappa}\Gamma(f) d\pi dt \\
&=
\frac{\sqrt{2}J}{\kappa^2}\int \Gamma(f) d\pi
\end{align*}
The conclusion follows by the duality fomula \eqref{duality} while taking the supremum when $g$ runs over all bounded continuous function on the left hand side of the last inequality.
\end{proof}

\section{Transport inequalities for Markov chain satisfying CDE'$(\kappa_e,\infty)$}
In this section, we assume that the exponential curvature condition CDE'$(\kappa_e,\infty)$ holds. We will prove Theorem \ref{thmW2I}. But first, we shall study some properties of the CDE'$(\kappa_e,\infty)$ condition.

\subsection{Properties of CDE'$(\kappa_e,\infty)$}
\begin{lem}\label{commutaionP_tsqrt}
Assue that CDE'$(\kappa$, $\infty)$ holds. Then for any nonnegative function $f: \cX \longrightarrow \R$ and any $t \geq 0$, we have
\begin{description}
\item $(i)$ $\Gamma(\sqrt{P_tf}) \leq e^{-2\kappa_e t}P_t\Gamma(\sqrt{f}).$
\item $(ii)$ $\frac{\Gamma(P_tf)}{P_tf} \leq e^{-2\kappa_e t}P_t\left(\frac{\Gamma(f)}{f}\right).$
\end{description}
\end{lem}
\begin{rque}
$(i)$ looks like the commutation formula of $\Gamma$ and $\sqrt{.}$ under $CD(\kappa,\infty)$ in continuous settings. But it is not the same thing, the positivity is very important. Recall that in classical Bakry-Emery theory, the commutation formula is the following:
 for all $f$(smooth enough), $\sqrt{\Gamma{(P_tf)}}\leq e^{-\kappa t}P_t(\sqrt {\Gamma(f)})$. We have not been able to recover this formula under CDE$(\kappa_e,\infty)$ in graphs settings.
\end{rque}

\begin{proof}
The proof follows a standard interpolation argument. We begin with $(i)$. Let $g:=P_{t-s}f$ and define $\phi(s) := e^{-2\kappa_e s}P_s\left(\Gamma(\sqrt{g})\right)$. To obtain the result, it is enough to show that $\phi' \geq 0$. In fact,
\begin{align*}
\phi'(s) &= e^{-2\kappa_e s}P_s\left[L(\Gamma(\sqrt{g})) - \Gamma\left(\sqrt{g}, \frac{Lg}{\sqrt{g}}\right) - 2\kappa \Gamma(\sqrt{g})\right] \\
&\geq 0,
\end{align*}
where we have used the assumption on the curvature, which is equivalent to
$$\frac{1}{2}L\Gamma(f) - \Gamma\left(f, \frac{L(f^2)}{2f}\right) \geq \kappa_e \Gamma(f)$$
(see (3.11) in \cite{BHLLMY}).

Similarly, let $\psi(s):= e^{-2\kappa_e s}P_s\left(\frac{\Gamma(g)}{g}\right)$. Again, it is enough to show that $\psi'\geq 0$. We have
\begin{equation}
\psi'(s)=e^{-\kappa_e s}P_s\left( L\left(\frac{\Gamma(g)}{g}\right)+\frac{1}{g^2}\left(-2g\Gamma(g,Lg)+\Gamma(g) Lg\right)-2\kappa \frac{\Gamma(g)}{g}\right).
\end{equation}
Since $g$ is positive, we only need to show that
\begin{equation}\label{eq6}
g\left(L\left(\frac{\Gamma(g)}{g}\right)+\frac{1}{g^2}\left(-2g\Gamma(g,Lg)+\Gamma(g) Lg\right)-2\kappa_e \frac{\Gamma(g)}{g}\right)\geq 0.
\end{equation}
Notice that \eqref{eq6} is equivalent to
$$gL\left(\frac{\Gamma(g)}{g}\right)-2\Gamma(g,Lg)+\frac{1}{g}\Gamma(g) Lg\geq 2\kappa_e \Gamma(g),$$
and we conclude by writing
$$ 2\kappa \Gamma(g)\leq 2\tilde{\Gamma}_2(g)=2\left(\Gamma_2(g)-\Gamma(g,\frac{\Gamma(g)}{g})\right)=gL\left(\frac{\Gamma(g)}{g}\right)-2\Gamma(g,Lg)+\frac{1}{g}\Gamma(g) Lg.$$

\end{proof}

\subsection{Weak transport-information inequalities under CDE'$(\kappa,\infty)$}
Using Lemma \ref{commutaionP_tsqrt}, we can prove some weak transport-information inequalities under CDE'$(\kappa_e,\infty)$. First, we will prove Theorem \ref{thmW2I}.

\begin{proof}[Proof of Theorem \ref{thmW2I}]
Let $\alpha(t)=e^{-\kappa_e t}$, for any probability density $f$ with respect to $\pi$ and any $t>0$, applying \eqref{HJalpha}, it holds:
\begin{equation}\label{eq5}
\int_0^{\infty}{\frac{d}{dt}\int{\tQ_{\alpha(t)}g P_tfd\pi}dt} \leq
\int_0^{\infty}{\int{-\frac{\kappa_e}{4}e^{-\kappa_e t}|\widetilde{\nabla} \tQ_{\alpha(t)}g|^2P_tf + \Gamma (\tQ_{\alpha (t)}g,P_t f) d\pi}dt}
\end{equation}
According to part $(iii)$ of Lemma \ref{Estimation of Gamma}, we have
\begin{align*}
\int{\tQ g fd\pi}& - \int{g d\pi} = \int_0^{\infty}{\frac{d}{dt}\int{\tQ_{\alpha(t)}g P_tfd\pi}dt} \\
&\leq
 \int_0^{\infty}{\int{-\frac{\kappa_e}{4}e^{-\kappa_e t}|\widetilde{\nabla} \tQ_{\alpha(t)}g|^2P_tf +2\sqrt{2J} |\widetilde{\nabla}\tQ_{\alpha(t)}g| \sqrt{P_tf\Gamma(\sqrt{P_t f})}d\pi} dt}\\
&\leq
 \int_0^{\infty}{\frac{8Je^{\kappa_e t}}{\kappa_e}\int{\Gamma (\sqrt{P_t f})d\pi}dt}.
\end{align*}
Now we apply Lemma \ref{commutaionP_tsqrt}, and we get
\begin{align*}
\int{\tQ g fd\pi} - \int{g d\pi}&\leq
 \int_0^{\infty}{\frac{8e^{-\kappa_e t}}{\kappa_e}\int{P_t\left(\Gamma (\sqrt{f})\right)d\pi}dt}\\
 &=
 \frac{8J}{\kappa_e^2}\int \Gamma (\sqrt{f})d\pi=\frac{2J}{\kappa_e^2}\I_\pi(f).
\end{align*}

The conclusion then follows from the duality formula \eqref{duality} by taking $\mu=\pi$ and $\nu=f\pi$.
\end{proof}
It is easy to see that $\I_\pi(f) \leq \overline{\I}_\pi(f):=\int \frac{\Gamma(f)}{f}d\pi$, thus we have the following corollary:
\begin{cor}\label{corWI}
Assume that the exponential curvature condition CDE'$(\kappa_e,\infty)$ holds, then $\pi$ satisfies the following weak-transport information inequalities:
$$\tW_2(f\pi|\pi)^2\leq \frac{2J}{\kappa_e^2}\overline{\I}_\pi(f).$$
\end{cor}
Unfortunately, we have not been able to establish the relation of $\tW_2(\pi|f\pi)^2$ and  $\int \Gamma (\sqrt{f})d\pi$. However, as in Corollary \ref{corWI}, we get a weaker inequality as follows:

\begin{thm}\label{W2Ifaible}
Assume that the exponential curvature condition CDE'$(\kappa_e,\infty)$ holds, then $\pi$ satisfies the following weak-transport information inequalities:
$$\tW_2(\pi|f\pi)^2\leq \frac{2J}{\kappa_e^2}\overline{\I}_\pi(f).$$
\end{thm}

\begin{proof}
We prove this theorem in a similar way as the previous one.

Let $\alpha(t):=1-e^{-\kappa t}$
Arguing as in the latter theorem, we get
\begin{align*}
\int{\tQ g d\pi}& - \int{g fd\pi} = \int_0^{\infty}{\frac{d}{dt}\int{\tQ_{\alpha(t)}g P_tfd\pi}dt} \\
&\leq
\int_0^{\infty}{\int{-\frac{\kappa}{4}e^{-\kappa_e t}|\widetilde{\nabla} \tQ_{\alpha(t)}g|^2P_tf - \Gamma (\tQ_{\alpha (t)}g,P_t f) d\pi}dt}
\end{align*}
Now applying part $(ii)$ of Lemma \ref{Estimation of Gamma}, it follows that
\begin{align*}
\int{\tQ g d\pi} &- \int{g fd\pi}
\leq
 \int_0^{\infty}{\int{-\frac{\kappa_e}{4}e^{-\kappa_e t}|\widetilde{\nabla} \tQ_{\alpha(t)}g|^2P_tf +\int |\widetilde{\nabla}\tQ_{\alpha(t)}g| \sqrt{2J\Gamma(P_t f)}d\pi} dt}\\
&\leq
 \int_0^{\infty}{\frac{2Je^{\kappa_e t}}{\kappa_e}\int{\frac{\Gamma (P_t f)}{P_tf}d\pi}dt} \leq
\int_0^{\infty}{\frac{2Je^{-\kappa_e t}}{\kappa_e}\int{P_t\left(\frac{\Gamma (f)}{f}\right)d\pi}dt} \\
& \hspace{1cm}= \frac{2J}{\kappa_e^2}\int{\frac{\Gamma(f)}{f}d\pi}
\end{align*}


\end{proof}

\begin{rque}
$(i)$ One can get Corollary \ref{corWI} by a similar argument:
let $\alpha(t):=e^{-\kappa t}$
\begin{align*}
\int{\tQ g fd\pi}& - \int{g d\pi} = \int_0^{\infty}{\frac{d}{dt}\int{\tQ_{\alpha(t)}g P_tfd\pi}dt} \\
&\leq
 \int_0^{\infty}{\int{-\frac{\kappa}{2}e^{-\kappa_e t}|\widetilde{\nabla} \tQ_{\alpha(t)}g|^2P_tf +\sqrt{2J} |\widetilde{\nabla}\tQ_{\alpha(t)}g| \sqrt{|\Gamma(P_t f)|}d\pi} dt}\\
&\leq
 \int_0^{\infty}{\frac{Je^{\kappa t}}{\kappa_e}\int{\frac{\Gamma (P_t f)}{P_tf}d\pi}dt} \leq
 \frac{J}{\kappa_e^2}\int{\frac{\Gamma (f)}{f}d\pi}dt.
\end{align*}

$(ii)$ Using the notations of \cite{GRST15}, define $\tW_2(f\pi,\pi)^2= \frac{1}{2}(\tW_2^2(f\pi|\pi)+\tW^2(\pi|f\pi))$, denote $\mathcal{P}_2(\cX)$ as the set of the probability measure on $\cX$ which has a finite second moment. Then $(\mathcal{P}_2(\cX),\tW_2(.,.))$ is a metric space, and if $\cX$ satisfies the exponential curvature condition, we have an upper bound for $\tW_2(.,\pi)$ in terms of modified Fisher information. Of course, when we work on a finite space, any probability measure has finite second moment.
\end{rque}

%
%
%
%

\section{Transport-information inequality for Markov chains with positive coarse Ricci curvature}
\label{sect_ollivier}
In this section, we assume the Markov chain has coarse Ricci curvature bounded from below by $\kappa_c$, with respect to the graph distance $d_g$.

As a consequence of the bound on the curvature, note that for any $1$-Lipschitz function $g$, $P_tg$ is $e^{-\kappa_c t}$-Lipschitz.


The problem of proving a transport-entropy inequality for Markov chains with positive coarse Ricci curvature was raised in Problem J in \cite{Oll2010}. It was proved by Eldan, Lee and Lehec \cite{ELL}. The transport-information inequality is a slight improvement of this result. Note that $T_1$ cannot hold in the full generality of the setting of \cite{Oll2009}, since it implies Gaussian concentration, which does not hold for some examples with positive curvature, such as Poisson distributions on $\N$.

The proof of this result will make use of the following lemma:
\begin{lem} \label{lem_transport-l1_d_g}
If the coarse Ricci curvature is bounded from below by $\kappa_c > 0$, then
$$W_1(f\pi, \pi) \leq \frac{1}{\kappa_c}\underset{x \neq y}{\sum} \hspace{1mm} |f(x) - f(y)|K(x,y)\pi(x).$$
\end{lem}

\begin{proof}
By Kantorovitch duality for $W_1$, we have
\begin{align*}
W_1(f\pi, \pi) &= \underset{g 1-lip}{\sup} \hspace{1mm} \int{gfd\pi} - \int{gd\pi} = \underset{g 1-lip}{\sup} \hspace{1mm} -\int_0^{+\infty}{\frac{d}{dt}\int{P_tg f d\pi}dt} \\
&= -\int_0^{+\infty}{\underset{x,y}{\sum} \hspace{1mm} (P_tf(y) - P_tg(x))(f(y) - f(x))K(x,y)\pi(x)dt} \\
&\leq \int_0^{+\infty}{||P_tg||_{lip}\underset{x,y}{\sum} \hspace{1mm} |f(y) - f(x)|K(x,y)\pi(x)dt} \\
&\leq \frac{1}{\kappa_c}\underset{x \neq y}{\sum} \hspace{1mm} |f(x) - f(y)|K(x,y)\pi(x).
\end{align*}
\end{proof}

We can now prove Theorem \ref{thmWIollivier}:
\begin{proof}[Proof of Theorem \ref{thmWIollivier}]

Observe that
\begin{align*}
&\sum_{x \neq y}\left(\sqrt{f}(x) + \sqrt{f}(y)\right)^2K(x,y)\pi(x)\\
&=\sum_{x \neq y}\left(2f(x)+2f(y)-\left(\sqrt{f}(x) - \sqrt{f}(y)\right)^2\right)K(x,y)\pi(x)\\
&\leq \sum_{x \neq y}(2f(x)+2f(y))K(x,y)\pi(x)-\sum_{x \neq y}\left(\sqrt{f}(x) - \sqrt{f}(y)\right)^2K(x,y)\pi(x)\\
&\leq 4J-\frac{1}{2}\I_\pi(f)
\end{align*}

Now using Lemma \ref{lem_transport-l1_d_g}, we have
\begin{align*}
W_1(f\pi, \pi) &\leq \frac{1}{\kappa_c}\underset{x \neq y}{\sum} \hspace{1mm} |f(x) - f(y)|K(x,y)\pi(x) \\
&= \frac{1}{\kappa_c}\underset{x \neq y}{\sum} \hspace{1mm} |\sqrt{f}(x) - \sqrt{f}(y)|\left(\sqrt{f}(x) + \sqrt{f}(y)\right)K(x,y)\pi(x) \\
&\leq \frac{1}{\kappa_c}\sqrt{\I_{\pi}(f)}\sqrt{\frac{1}{4}\underset{x \neq y}{\sum} \hspace{1mm} \left(\sqrt{f}(x) + \sqrt{f}(y)\right)^2K(x,y)\pi(x)} \\
&\leq \frac{1}{\kappa_c}\sqrt{\I_{\pi}(f)}\sqrt{J-\frac{1}{8}\I_\pi(f)}.
\end{align*}
\end{proof}

\section{Applications}
\subsection{Transport-information inequalities implies transport-entropy inequalities}
 We prove here the discrete version of Theorem 2.1 in \cite{GLWW2}. The proof is essentially unchanged, we give it to justify the validity of the theorem in the discrete setting.
\begin{thm}
Assume that the transport-information inequality
$$W_1(f\pi, \pi)^2 \leq \frac{1}{C^2}\mathcal{I}_{\pi}(f)$$
holds. Then we have the transport-entropy inequality
$$W_1(f\pi, \pi)^2 \leq \frac{2}{C}\Ent_{\pi}(f).$$
\end{thm}

\begin{proof}
The transport-entropy inequality $T_1(C)$ is equivalent to the estimate
$$\int{e^{\lambda f}d\pi} \leq \exp\left(\frac{\lambda^2}{2C}\right)$$
for all $1$-Lipschitz function $f$ with $\int{fd\pi} = 0$ and all $\lambda \geq 0$. Let $f$ be such a function. Let $Z(\lambda) := \int{e^{\lambda f}d\pi}$ and $\mu_{\lambda} := e^{\lambda f}\pi/Z(\lambda)$. We have
\begin{equation*}
\frac{d}{d\lambda}\log Z(\lambda) = \frac{1}{Z(\lambda)}\int{fe^{\lambda f}d\pi} \leq W_{1,d}(\mu_{\lambda}, \pi) \leq  \sqrt{\frac{4}{C^2}\int{\frac{\Gamma(e^{\lambda f/2})}{Z(\lambda)}d\pi}}.
\end{equation*}
Using the inequality $\int{\Gamma(f)d\pi} \leq \int{f^2\Gamma(\log f)d\pi}$, we deduce
$$\frac{d}{d\lambda} \log Z(\lambda) \leq \frac{\lambda}{C}$$
which integrates into $\log Z(\lambda) \leq \lambda^2/(2C)$, and this is the bound we were looking for.
\end{proof}

We shall now show that the weak transport-information inequality $\widetilde{T}^+_2I$ implies the weak-transport-entropy inequality $\widetilde{T}_2^+H$. The proof is an adaptation of the one for the $T_2$ and $T_2I$ inequalities in the continuous setting from \cite{GLWW2}.

\begin{thm}
Assume that $\pi$ satisfies the modified weak-transport information inequality $\widetilde{T}_2^+I(C)$, then $\pi$ satisfies the weak-transport inequality $\widetilde{T}_2^+H(C)$.
\end{thm}

\begin{proof}
According to \cite{GRST15}, the transport-entropy inequality $\widetilde{T}_2$ is equivalent to
$$\int{\exp\left(\frac{2}{C}\tQ_1 f\right)d\pi} \leq \exp\left(\frac{2}{C}\int{fd\pi}\right) \hspace{3mm} \forall f : \cX \longrightarrow \R \hspace{2mm} \text{bounded}.$$
Usually, the class of functions $f$ we must use is the class of bounded continuous functions, but here, since we work on a discrete space endowed with a graph distance, we only have to work with bounded functions.

We write $F(t) := \log \int{\exp(k(t)\tQ_t f)d\pi} -k(t)\int{fd\pi}$ with $k(t) :=C t $. Let $\mu_t$ be the probability measure with density with respect to $\pi$ proportional to $\exp(k(t)\tQ_t f)$.

According to part (iv) of Lemma \ref{Estimation of Gamma}, we have
$$\int{\Gamma(e^{\frac{1}{2}(k(t)\tQ_tf)})d\pi}\leq \int{|\widetilde{\nabla}k(t)\tQ_tf|^2d\mu_t}$$
Hence, for $t > 0$, we have
\begin{align*}
F'(t) &\leq \frac{1}{\int{\exp(k(t)\tQ_t f)d\pi}}\left(\int{k'(t)\tQ_t f e^{k(t)\tQ_tf}d\pi} - \int{k(t)|\widetilde{\nabla}\tQ_t f|^2 e^{k(t)\tQ_tf}d\pi}\right) - k'(t)\int{fd\pi} \\
&\leq \frac{k'(t)}{t}\left(\int{\tQ_1(tf)d\mu_t} - \int{tfd\pi}\right) - k(t)\int{|\widetilde{\nabla}\tQ_tf|^2d\mu_t} \\
&\leq \frac{k'(t)}{t}\tW_2(\mu_t, \pi)^2 - k(t)\int{|\widetilde{\nabla}\tQ_tf|^2d\mu_t} \\
&\leq \frac{k'(t)}{tC^2}\int{\Gamma(e^{\frac{1}{2}(k(t)\tQ_tf)})d\pi} - k(t)\int{|\widetilde{\nabla}\tQ_tf|^2d\mu_t} \\
&\leq \left(\frac{k'(t)}{tC^2}-\frac{1}{k(t)}\right)\int{|\widetilde{\nabla}k(t)\tQ_tf|^2d\mu_t}\\
&= 0
\end{align*}
\end{proof}

\subsection{Transport-information inequalities imply diameter bounds}
We now show that transport-information inequalities imply diameter bounds, in the spirit of the $L^1$ Bonnet-Meyer theorem of \cite{Oll2009}.
\begin{thm}[Diameter estimate]
Assume that the transport-information inequality
$$W_{1,d}(f\pi, \pi)^2 \leq \frac{1}{C^2}\mathcal{I}_{\pi}(f)$$
holds, for some distance $d$. Then
$$\sup_{x,y\in \cX}d(x,y)\leq \frac{2}{C}\left(\sqrt{J(x)}+\sqrt{J(y)}\right) $$
\end{thm}
\begin{proof}
Let $f_x$  be the density function correspond to the Dirac mass $\delta_x$: $f_x:=\frac{1_{x}}{\pi(x)}$.
\begin{align*}
\int \Gamma(\sqrt{f_x})d\pi
&=
\frac{1}{2}\left(\sum_{z\sim x}(\sqrt{f_x(x)})^2k(x,z)\pi(x)+\sum_{z\sim x}(\sqrt{f_x(x)})^2k(z,x)\pi(z)\right)\\
&=
\sum_{z\sim x}(\sqrt{f_x(x)})^2k(x,z)\pi(x)
\leq J(x)
\end{align*}
Thus, for all $x\in \cX$, $\I_\pi(f_x)\leq 4J(x)$, it follows that $W_1(f_x\pi,\pi)^2\leq \frac{4}{C^2}J(x) $.
As a consequence, for all $x,y\in \cX$, it holds
$$d(x,y)=W_1(\delta_x,\delta_y)\leq W_1(f_x\pi,\pi)+W_1(f_y\pi,\pi)\leq \frac{2}{C}\left(\sqrt{J(x)}+\sqrt{J(y)}\right).$$
\end{proof}
According to Corollary \ref{corW2ICD}, we have the following corollary:
\begin{cor}
Assume that $CD(\kappa,\infty)$ holds, then for all $x,y\in \cX$, $$d_{\Gamma}(x,y)\kappa \leq 2\sqrt{2}\left(\sqrt{J(x)}+\sqrt{J(y)}\right).$$
\end{cor}
Together with \eqref{d_and_d_gamma2}, we can recover Corollary\ref{dia_estimate}.

If we look at the example of the discrete hypercube of dimension $N$ (see the next subsection), with our notations it satisfies CD$(1/N,\infty)$. The above theorem gives the correct bound on the diameter for the graph distance of $N$.

\subsection{An example: the discrete hypercube}
As an example of Markov chain satisfying CDE'$(\kappa$, $\infty$), we study the example of the symmetric random walk on the discrete hypercube. It is a Markov chain on $\{0,1\}^N$, which at rate $1$ selects a coordinate uniformly at random, and flips it with probability $1/2$. The transition rates are $K(x,y) = 1/(2N)$ for $x,y$ such that $d_g(x,y) = 1$, and else it is $0$.

\begin{thm}
The symmetric random walk on the discrete hypercube satisfies CDE'$(1/N$, $\infty)$
\end{thm}

\begin{proof}
We start with the case $N = 1$. Since then we only have to consider a Markov chain on a two-points space, we can easily do explicit computations. Fix $f : \{0,1\} \longrightarrow \R$. We have
$$\Gamma(f)(0) = \Gamma(f)(1) = \frac{1}{4}(f(0)-f(1))^2$$
and hence $\Gamma\left(f, \frac{\Gamma(f)}{f}\right) = 0$ and
$$\tilde{G}_2(f) = \Gamma_2(f) = -\Gamma(f, Lf) = \Gamma(f).$$
Therefore when $N = 1$, the Markov chain satisfies CDE'$(1$, $\infty)$.

The general case follows, using a tensorization argument. In the unnormalized case, using Proposition 3.3 of \cite{BLLY}, the graph satisfies CDE'$(1$, $\infty)$ independently of $N$. Since we consider the case of a Markov chain and enforce \eqref{assumpt_markov}, we rescale the generator by a factor $1/N$ (so that there is on average one jump by unit of time), and therefore it satisfies CDE'$(1/N$, $\infty)$.
\end{proof}

\begin{rque}
We have shown that for the two-point space, the exponential curvature and the curvature are the same, and equal to $1$. In \cite{KKRT14}, it is stated that the curvature is $2$. The difference is because, since we enforced the normalization condition \eqref{assumpt_markov}, the definitions of $L$ in the two frameworks differ by a factor $2$.
\end{rque}

\section*{Acknowledgement}
M.F. thanks Matthias Erbar, Ivan Gentil and Prasad Tetali for discussions on this topic. This work was partly done during a stay at the Hausdorff Research Institute for Mathematics in Bonn, whose support is gratefully acknowledged. 

Y.S. thanks his PhD advisors Natha\"el Gozlan and Cyril Roberto for helpful advice and remarks.

\end{document}